%% file: main.tex
\documentclass[12pt]{article}

\usepackage[utf8]{inputenc}
\usepackage{lmodern}
\usepackage{amsmath,amssymb}
\usepackage{amsthm}
\usepackage{xcolor}
\usepackage{setspace}
\usepackage[margin=2.54cm]{geometry}
\usepackage{graphicx}
\usepackage{mathtools}
\usepackage[ruled]{algorithm2e}
\usepackage{hyperref}

\newcommand{\tr}{\operatorname{tr}}

\newcommand{\R}{\mathbb{R}}
\newcommand{\C}{\mathbb{C}}

\newcommand{\F}{\mathcal{F}}

\newcommand{\G}{\mathcal{G}}

\newcommand{\SOC}{\text{SOC}}
\newcommand{\Shor}{{\cal R}_{\text{shor}}}

\newcommand{\RLT}{{\cal R}_{\text{rlt}}}
\newcommand{\SOCRLT}{{\cal R}_{\text{socrlt}}}
\newcommand{\KSOC}{{\cal R}_{\text{ksoc}}}

\newcommand{\Rc}{{\cal R}}
\newcommand{\Cuts}{{\cal R}_{\text{cuts}}}

\renewcommand{\Re}{\text{Re}}
\renewcommand{\Im}{\text{Im}}

\renewcommand{\nu}{R}
\renewcommand{\gamma}{r}
\renewcommand{\alpha}{a}
\renewcommand{\rho}{[c]_{\max}}

\newcommand{\qllowbd}{[q+l]_{\mathrm{lowbd}}}

\newcommand{\st}{\text{s.t.}}

\newcommand{\clconv}[1]{\text{conv}\left\lbrace #1 \right\rbrace}

\newtheorem{theorem}{Theorem}
\newtheorem{corollary}{Corollary}
\newtheorem{lemma}{Lemma}
\newtheorem{proposition}{Proposition}

\newtheorem{example}{Example}

\newtheorem*{observation}{Observation}
\newtheorem{conjecture}{Conjecture}

\author{Anders Eltved\thanks{Department of Applied Mathematics and Computer Science,
        Technical University of Denmark, 2800 Kgs. Lyngby, Denmark. Email:
    {\tt anderseltved@gmail.com}.} \and Samuel Burer%
\thanks{Department of Business Analytics, University of Iowa,
Iowa City, IA, 52242-1994, USA. Email: {\tt samuel-burer@uiowa.edu}.}%
}

\date{September 26, 2020 \\ Revised: July 26, 2021}

\title{Strengthened SDP Relaxation for an \\ Extended Trust Region
Subproblem \\ with an Application to Optimal Power Flow}

\begin{document}

\maketitle

\input{abstract}

\input{introduction}

\input{valid_inequalities}

\input{applications}

\input{separation}

\input{numerical_examples}

\input{conclusions}

\bibliographystyle{abbrv}
\bibliography{main}

\singlespacing

\end{document}

%% file: abstract.tex
\begin{abstract}
We study an extended trust region subproblem minimizing a nonconvex
function over the hollow ball $r \le \|x\| \le R$ intersected with a
full-dimensional second order cone (SOC) constraint of the form $\|x -
c\| \le b^T x - a$. In particular, we present a class of valid cuts
that improve existing semidefinite programming (SDP) relaxations and are
separable in polynomial time. We connect our cuts to the literature on
the optimal power flow (OPF) problem by demonstrating that previously
derived cuts capturing a convex hull important for OPF are actually
just special cases of our cuts. In addition, we apply our methodology
to derive a new class of closed-form, locally valid, SOC cuts for
nonconvex quadratic programs over the mixed polyhedral-conic set $\{x
\ge 0 : \| x \| \le 1 \}$. Finally, we show computationally on randomly
generated instances that our cuts are effective in further closing the
gap of the strongest SDP relaxations in the literature, especially in
low dimensions.
\end{abstract}

%% file: introduction.tex
\section{Introduction}

The classical {\em trust region subproblem (TRS)\/} minimizes an
arbitrary quadratic function over the unit Euclidean ball defined by
$\|x\| \le \nu$ and is solvable in polynomial-time \cite{Conn2000}.
Many authors have studied variants of TRS that incorporate additional
constraints. For example, \cite{Stern1995} also imposes the lower bound
$\gamma \le \|x\|$. We collectively refer to variants of TRS that
incorporate more general constraints as the {\em extended TRS\/}. In
this paper, we study the following specific form of the extended TRS,
which incorporates the lower bound $\gamma$ as well as an additional SOC
(second-order cone) constraint, whose ``geometry'' matches the ball in
the sense that its Hessian is also the identity matrix:
\begin{subequations} \label{equ:qcqp_main}
\begin{align}
    \min \ \ \ &x^T H x + 2 \, g^T x \\
\st \ \ \ \ &\gamma \le \|x\| \le \nu \label{sub:qcqp_cons_quad} \\
            &\| x - c \| \le b^Tx - \alpha \label{sub:qcqp_cons_soc}
\end{align}
\end{subequations}
where $x \in \R^n$, $H = H^T \in \R^{n \times n}$, $g, c, b \in \R^n$,
$\alpha \in \R$, and $\gamma, \nu \in \R_+$. Note that $H$ is symmetric
without loss of generality and that we have {\em not\/} scaled the
problem to the unit ball (i.e., we do not assume $R = 1$) as is common
in the TRS literature. The general upper bound $\nu$ will be convenient
for our presentation, especially in Section \ref{sec:applications}. The
algorithm of Bienstock \cite{Bienstock2016} solves (\ref{equ:qcqp_main})
in polynomial time since it can be written as a nonconvex quadratic
program with a fixed number of quadratic/linear constraints (in
this case, four), one of which is strictly convex. However, in this
paper, we are interested in developing tight convex relaxations
of (\ref{equ:qcqp_main}).
In particular, as far as we are aware, (\ref{equ:qcqp_main}) has 
no known tight convex relaxation.

Problem (\ref{equ:qcqp_main}) includes, for example, a special case of the {\em
two trust region subproblem\/}---also called the {\em Celis-Dennis-Tapia
subproblem\/} \cite{Celis1985}---in which a second ball 
constraint is added to TRS. In this case, $r = 0$, $b = 0$, and
$\alpha < 0$. Here, however, we are interested in the more general
structure represented by (\ref{sub:qcqp_cons_soc}), which arises,
for example, in the {\em optimal power flow problem (OPF)\/} as
discussed in Section \ref{sec:applications}. More generally, the study
of (\ref{equ:qcqp_main}) sheds light on any nonconvex quadratically
constrained quadratic program that includes a ball constraint
and a second SOC constraint with identity Hessian. In Section
\ref{sec:applications}, we will also show how this structure is relevant
for the mixed polyhedral-SOC set $\{ x \ge 0 : \|x\| \le R \}$. (In
the concluding Section \ref{sec:conclusions}, we briefly mention an
extension for handling different Hessians.)

Since (\ref{equ:qcqp_main}) is a nonconvex problem, a standard approach
is to approximate (\ref{equ:qcqp_main}) by its so-called {\em Shor
semidefinite programming (SDP) relaxation\/} \cite{Shor1987}, which is
solvable in polynomial time:
\begin{subequations} \label{equ:Shor_main}
\begin{align}
\min \ \ \ &H \bullet X + 2 \, g^T x \\
\st \ \ \ \ &\gamma^2 \le \tr(X) \le \nu^2 \label{sub:Shor_main_trace} \\
            &\tr(X) - 2 \, c^T x + c^T c
            \le bb^T \bullet X - 2 \, \alpha \, b^T x + \alpha^2 \label{sub:Shor_main_soc_1} \\
            &0 \le b^T x - \alpha \label{sub:Shor_main_soc_2} \\
            &Y(x,X) \succeq 0 \label{sub:Shor_psd}
\end{align}
\end{subequations}
where $M \bullet X := \tr(M^T X)$ is the trace inner product for
conformal matrices and
\begin{equation} \label{equ:Ydefn}
    Y(x,X) := \begin{pmatrix} 1 & x^T \\ x & X \end{pmatrix}
\end{equation}
is symmetric of size $(n + 1) \times (n + 1)$.
Note that (\ref{sub:qcqp_cons_soc}) is represented
as the two constraints $\| x - c \|^2 \le (b^T x -
\alpha)^2$ and $0 \le b^T x - \alpha$ before lifting to
(\ref{sub:Shor_main_soc_1})--(\ref{sub:Shor_main_soc_2}). We also define
\[
    \Shor := \{ (x,X) : (x,X) \text{ satisfies
    (\ref{sub:Shor_main_trace})--(\ref{sub:Shor_psd})} \}
\]
to be the feasible set of the Shor relaxation. Then
(\ref{equ:Shor_main}) can be alternatively expressed as minimizing $H
\bullet X + 2 \, g^T x$ over $(x,X) \in \Shor$.

Various valid inequalities can be added to (\ref{equ:Shor_main}) in
order to strengthen the Shor relaxation. For example, if $v_1^T x \ge
u_1$ and $v_2^T x \ge u_2$ are any two valid linear inequalities for
the feasible set of (\ref{equ:qcqp_main}), then the redundant quadratic
constraint $(v_1^T x - u_1) (v_2^T x - u_2) \ge 0$ can be relaxed to
the valid {\em RLT constraint\/} \cite{Sherali1999}:
\[
    v_1 v_2^T \bullet X - u_2 v_1^T x - u_1 v_2^T x + u_1 u_2
    \ge 0.
\]
However, since (\ref{equ:qcqp_main}) does not contain explicit linear
constraints, in practice one would need to separate over valid $v_1^T x
\ge u_1$ and $v_2^T x \ge u_2$ to generate violated RLT constraints, but
this separation is a bilinear subproblem, which does not appear to be
solvable in polynomial time.

The difficulty of separating the RLT constraints when no linear
constraints are explicitly given can be circumvented in the case of
(\ref{equ:qcqp_main}) as follows. By multiplying a valid $v_1^T x \ge
u_1$ with the ball constraint $\|x\| \le \nu$, we have the redundant
quadratic SOC constraint $\| (v_1^T x - u_1) x \| \le R (v_1^T x -
u_1)$, which in turn yields the valid SOC constraint
\begin{equation} \label{equ:socrlt}
    \| X v_1 - u_1 x \| \le \nu(v_1^T x - u_1)
\end{equation}
in the lifted $(x,X)$ space. In a similar manner, $v_1^T x \ge u_1$
can be combined with $\|x - c\| \le b^T x - \alpha$. These are known
as {\em SOCRLT constraints\/} \cite{Sturm2003, Ye2003, Burer2013}. In fact, each
SOCRLT constraint is a compact encoding of an entire collection of
RLT constraints. For example, (\ref{equ:socrlt}) captures all of the
RLT constraints corresponding to $v_1^T x \ge u_1$ fixed and $v_2^T
x \ge u_2$ varying over the supporting hyperplanes of $\|x\| \le
R$. Consequently, the collections of SOCRLT and RLT constraints for
(\ref{equ:qcqp_main}) are equivalent,\footnote{This differs from other
papers, which often define RLT constraints only for explicitly
given valid linear constraints, of which (\ref{equ:qcqp_main})
has none. So, for the sake of generality, we have defined the RLT
constraints allowing for {\em implicit\/} valid linear constraints.}
but in contrast to the RLT constraints, the SOCRLT constraints
can be separated in polynomial-time based on the fact that TRS is
polynomial-time solvable \cite{Burer2013}.

Anstreicher \cite{Anstreicher2017} introduced a further generalization
of the SOCRLT constraints, called a {\em KSOC constraint\/}, which is
based on relaxing a valid quadratic Kronecker-product matrix inequality.
Specifically, the KSOC constraint is constructed from the following
observations: first, defining $\SOC := \{ (v_0, v) : \|v\| \le v_0 \}$
to be the second-order cone, it is well-known that
\[
    {v_0 \choose v} \in \SOC \ \ \Longleftrightarrow \ \
    \begin{pmatrix}
        v_0 & v^T \\ v & v_0 I
    \end{pmatrix} \succeq 0;
\]
second, it is also well-known that the Kronecker product of positive
semidefinite matrices is positive semidefinite. Hence, for (\ref{equ:qcqp_main}) we have the valid
quadratic matrix inequality
\[
    \begin{pmatrix}
        \nu & x^T \\ x & \nu \, I
    \end{pmatrix}
    \otimes
    \begin{pmatrix}
        b^T x - \alpha & x^T - c^T \\ x - c & (b^T x - \alpha) I
    \end{pmatrix}
    \succeq 0.
\]
After relaxing this inequality in the space $(x,X)$, we obtain the
convex KSOC constraint, which captures all SOCRLT constraints (and hence
all RLT constraints) and is generally stronger \cite{Anstreicher2017},
assuming the Shor constraints remain enforced.

Summarizing, defining $\RLT$ and $\SOCRLT$ to be the set of $(x,X)$
satisfying all possible RLT and SOCRLT constraints, respectively, 
we have
\[
    \Shor \cap \KSOC
    \subseteq
    \Shor \cap \SOCRLT
    =
    \Shor \cap \RLT
\]
where $\KSOC$ is the set of all $(x,X)$ satisfying the KSOC
constraint. Moreover, the first containment is proper in general.
Hence, in this paper, we focus on improving
the relaxation $\Shor \cap \KSOC$. The paper \cite{Jiang2019} provides
further insight into the strength of $\Shor \cap \KSOC$ relative to
other techniques in the literature.

Let $\F$ denote the feasible set of (\ref{equ:qcqp_main}),
i.e., the set of all $x \in \R^n$ satisfying
(\ref{sub:qcqp_cons_quad})--(\ref{sub:qcqp_cons_soc}). Strengthening
the SDP relaxation can alternatively be expressed as determining valid
inequalities that more accurately approximate the closed convex hull
\begin{equation} \label{equ:Gdefn}
    \G := \clconv{(x,xx^T) : x \in \F }.
\end{equation}
Note that $\G$ is compact because $\F$ is. Moreover, because linear
optimization over a compact convex set is guaranteed to attain its optimal
value at an extreme point, solving (\ref{equ:qcqp_main}) amounts to
optimizing the linear function $H \bullet X + 2 \, g^T x$ over $\G$ .
While an exact representation of $\G$ is unknown, there are several
closely related cases in which $\G$ can be described exactly; see
\cite{Burer2015,Argue2020}.

In this paper, we propose a new class of valid linear inequalities for
\eqref{equ:qcqp_main} in the space $(x,X)$, which in general strengthen
$\Shor \cap \KSOC$ towards $\G$. 
Each inequality is derived from several ingredients
that exploit the structure of $\F$:  the self-duality of SOC; the RLT-type
valid inequality $(\nu - \|x\|)(\|x\| - \gamma) \ge 0$; and knowledge of
a quadratic function $q(x)$ and a linear function $l(x)$, each of which
is nonnegative over all $x \in \F$. We combine these ingredients to derive
a valid quartic inequality, which is then relaxed to a valid quadratic
inequality, which in turn yields a new valid linear inequality in
$(x,X)$.

As a small illustrative example, consider when \( c = 0 \) and $r =
0$, in which case $\F$ is defined by $\|x\| \le R$ and $\|x\| \le b^T
x - \alpha$. For the specific choices $q(x) = 0$ and $l(x) = 1$, our
new inequality can also be derived from the following direct argument:
the chain of inequalities $\|x\|^2 \le R \|x\| \le R(b^T x - \alpha)$
linearizes to
\begin{equation} \label{equ:basicineq0}
    \tr(X) \le R(b^T x - \alpha).
\end{equation}
The following example shows that (\ref{equ:basicineq0}) is not captured
by \( \Shor\cap\KSOC \):
\begin{example} \label{exa:orthog}
Let $\F = \{ x \in \R^2 : \|x\| \le 1, \|x\| \le 1 - x_1 - x_2 \}$. Then
\eqref{equ:basicineq0} is $\tr(X) \le 1 - x_1 - x_2$. Minimizing the
objective $1 - x_1 - x_2 - \tr(X)$ over $\Shor \cap \KSOC$ yields the
optimal solution
\[ Y^* \approx 
    \begin{pmatrix*}[r]
  1.0000 & 0.0624 &  0.0624 \\
  0.0624 & 0.5000 &  -0.3018 \\
  0.0624 & -0.3018 &   0.5000
\end{pmatrix*}
\]
with (approximate) optimal value \( -0.1248 \), i.e., the optimal value is negative,
which demonstrates that \eqref{equ:basicineq0} is not valid for $\Shor
\cap \KSOC$.
\end{example}

\noindent As far as we are aware, inequality (\ref{equ:basicineq0}) for
this special case has not yet appeared in the literature. We seek in
this paper, however, an even more general procedure for deriving valid
inequalities using the ingredients described in the previous paragraph.

The paper is organized as follows. In Section \ref{sec:new_ineqs},
we present the derivation of our new valid inequalities and discuss
several illustrative choices of $q(x)$ and $l(x)$. We also specialize
the results to $c = 0$ and $a = 0$, a case which further enables the
derivation of a similar, second type of valid linear inequality in
$(x,X)$. Then, in Section \ref{sec:applications}, we show that our
inequalities include those introduced in \cite{Chen2017} for the study
of the OPF problem,\footnote{Indeed, our initial motivation for this
paper was the desire to understand the inequalities in \cite{Chen2017}
more fully.} and we extend our approach to derive a new class of
valid SOC constraints for $\G$ when $\F$ equals the intersection
of the ball $\|x\| \le R$ and the nonnegative orthant. Next, in
Section \ref{sec:separation}, we prove that the separation problem for
our inequalities---which can be viewed as dynamically choosing the
nonnegative functions $q(x)$ and $l(x)$---is polynomial-time based
on the availability of any SDP relaxation in the variables $(x,X)$,
such as the relaxations $\Shor$ or $\Shor \cap \KSOC$. In this sense,
we are able to ``bootstrap'' any existing SDP relaxation for the
separation subroutine to generate valid cuts. Finally, in Section
\ref{sec:computation}, we provide computational evidence that our cuts
are effective in further closing the gap between (\ref{equ:qcqp_main})
and $\Shor \cap \KSOC$ on randomly generated problems, especially in low
dimensions. We close in Section \ref{sec:conclusions} with a few final
thoughts and directions for future research.

This paper is accompanied by the code repository
\url{https://github.com/A-Eltved/strengthened_sdr}, which contains full
code for the paper's examples and computational results. In addition,
the first author's forthcoming Ph.D.~thesis \cite{Eltved2021} will
contain additional discussion and extensions.

%

%% file: valid_inequalities.tex
\section{New Valid Inequalities} \label{sec:new_ineqs}

In the Introduction, we discussed the valid inequality
(\ref{equ:basicineq0}) for the specific case $c = 0$ and $r = 0$.
Now we assume general $c$ and $r$. Analogous to
(\ref{equ:basicineq0}), we use $\|x\| \le \nu$ and $\|x-c\| \le b^T x -
\alpha$ along with the self-duality of $\SOC$ to obtain the following
quadratic inequality:
\begin{equation} \label{equ:basicineq}
    {\nu \choose -x}^T {b^Tx - \alpha \choose x - c} \ge 0
    \quad \Longrightarrow \quad
    \nu (b^T x - \alpha) \ge \tr(X) - c^T x.
\end{equation}
Note that this inequality makes use of the equivalent
constraint $\|-x\| \le \nu$. We seek to strengthen it further by
incorporating two additional ideas.

The first idea involves exploiting the lower bound $\gamma \le \|x\|$
and the RLT-type valid inequality $(R - \|x\|)(\|x\| - r) \ge 0$.
Consider the following proposition:
\begin{proposition} \label{pro:scaledx}
Suppose $\gamma \le \|x\| \le \nu$, and define
$\gamma \|x\|^{-2} := 0$ when $\|x\| = \gamma = 0$. Then
\begin{equation} \label{equ:scaledx}
    \begin{pmatrix}
        \gamma + \nu \\ \left( 1 + \gamma \nu \|x\|^{-2} \right) x
    \end{pmatrix}
    \in \SOC.
\end{equation}
\end{proposition}
\begin{proof}
If $\gamma = 0$, then (\ref{equ:scaledx}) reads $(\nu, x) \in
\SOC$, which is true by assumption. So suppose $0 < \gamma \le \|x\|$.
Then we wish to prove
\[
    (1 + \gamma \nu \|x\|^{-2}) \|x\| = \|x\| + \gamma \nu \|x\|^{-1}
    \le \gamma + \nu,
\]
which follows by expanding the valid expression \( (\nu - \|x\|)(\|x\| -
\gamma) \ge 0 \) and dividing by $\|x\| \ge \gamma > 0$.
\end{proof}
\noindent By the proposition, analogous to (\ref{equ:basicineq}), we
have:
\begin{align*}
    &{\gamma + \nu \choose -(1 + \gamma \nu \|x\|^{-2})x}^T {b^Tx - \alpha \choose x - c} \ge 0 \\
    &\quad\quad \Longleftrightarrow
    (\gamma + \nu) (b^T x - \alpha) \ge x^T x + \gamma \nu - c^T x - \gamma \nu \|x\|^{-2} \, c^T x.
\end{align*}
However, this inequality cannot be directly linearized in $(x,X)$
due to the non-quadratic term $\|x\|^{-2}$. So we bound the term
$\gamma\|x\|^{-2} \, c^T x$ from above by a problem-dependent constant
$\rho \ge 0$, which satisfies $\gamma \, c^T x \le \rho \, x^T x$ for
all $x \in \F$. We then have the valid linear inequality
\begin{equation} \label{equ:basicineq2}
    (\gamma + \nu) (b^T x - \alpha) \ge \tr(X) + \gamma \nu - c^T x - \rho\nu.
\end{equation}
Such a $\rho$ clearly exists. For example, $\rho = \|c\|$ works because
\[
    \gamma \, c^T x \le \gamma \|c\| \|x\| \le \|c\| \|x\|^2,
\]
but naturally it is advantageous to take $\rho$ as small as possible.
One method for computing a smaller $\rho \le \|c\|$ is binary search
on $\rho$ over the interval $[0, \|c\|]$, where at each step we check
whether the optimal value of
\[
    \min_x \left\{
        \rho \, x^T x - \gamma \, c^T x : \|x\| \le \nu, \|x - c\| \le b^T x - \alpha
    \right\}
\]
is nonnegative. The nonconvex lower bound $\gamma \le \|x\|$
has been excluded from this subproblem to ensure convexity and
polynomial-time solvability, which also ensures that the binary search
is polynomial-time overall. Note that the binary search will not always return
the smallest possible $\rho$ due to the exclusion of the lower bound.
Note also that, when $\gamma = 0$ or $c = 0$, the optimal $\rho$ equals $0$.

Our second idea to improve (\ref{equ:basicineq}) and
(\ref{equ:basicineq2}) is to replace $(b^T x - \alpha, x - c) \in \SOC$
in the derivation above with another vector---but one that is still
in the second-order cone. In particular, we consider the nonnegative
combination
\begin{equation} \label{equ:mult_soc}
    q_x {\nu \choose x} + l_x {b^T x - \alpha \choose x - c} \in \SOC,
\end{equation}
where $q_x := q(x)$ is a quadratic function and $l_x := l(x)$ is
a linear function, both of which are nonnegative for all $x \in
\F$. This approach is similar to polynomial-optimization approaches
such as the one pioneered in \cite{Lasserre2000/01}, which uses
polynomial multipliers with limited degree to derive new, albeit
redundant, constraints. Then we have the following generalization of
(\ref{equ:basicineq2}):
\[
    {\gamma + \nu \choose -(1 + \gamma \nu \|x\|^{-2})x}^T {\nu q_x + l_x (b^Tx - \alpha) \choose (q_x + l_x) x - l_x c} \ge 0
\]
which rearranges and relaxes to
\[
    (\gamma + \nu) \nu q_x + (\gamma + \nu) l_x (b^T x - \alpha)
    \ge \left( q_x + l_x \right) x^T x + \gamma \nu \left( q_x + l_x\right) - l_x c^T x - \rho \, \nu \, l_x.
\]
Note that the right-hand side is quartic in $x$, and hence this inequality cannot
be directly linearized in the space $(x,X)$. Hence, we define a constant that
satisfies
\[
    \min \{ q_x + l_x : x \in \F \} \ge \qllowbd \ge 0
\]
to get the valid quadratic inequality
\begin{equation} \label{equ:general_inequality}
    (\gamma + \nu) \nu q_x + (\gamma + \nu) l_x (b^T x - \alpha)
    \ge \qllowbd \, x^T x + \gamma \nu \left( q_x + l_x\right) - l_x c^T x - \rho \nu \, l_x,
\end{equation}
which can be easily linearized in $(x,X)$ as summarized in the
following theorem. Note that the theorem requires only 
that $\qllowbd$ be a nonnegative lower bound on the value of $q(x) +
l(x)$ over $\F$, but generally a larger value gives a tighter valid inequality.

\begin{theorem} \label{the:main}
Let $\F$ be the feasible set of (\ref{equ:qcqp_main}), and let $\rho \in
[0,\|c\|]$ be given such that $\gamma \, c^T x \le \rho x^T x$ for all
$x \in \F$. In addition, let $q(x) := x^T H_q x + 2 \, g_q^T x + f_q$
and $l(x) := 2 \, g_l^T x + f_l$ be given such that $q(x) \ge 0$ and $l(x)
\ge 0$ for all $x \in \F$. Also, let $\qllowbd \ge 0$ be a valid
lower bound on the sum $q(x) + l(x)$ over all $x \in \F$. Then the
linear inequality
\begin{align}
(\gamma + \nu)&\nu \left(H_q \bullet X + 2 \, g_q^T x + f_q \right)
+ (\gamma + \nu) \left( 2 \, g_l b^T \bullet X + (f_l b - 2 \alpha g_l)^T x - \alpha f_l \right) \nonumber \\
&\ge \qllowbd \tr(X) + \gamma\nu \left( H_q \bullet X + 2 (g_q + g_l)^T x + (f_q + f_l) \right) \nonumber \\
& \quad\quad - \left( 2 \, g_l c^T \bullet X + f_l c^T x \right) - \rho \nu (2 \, g_l^T x + f_l) \label{equ:main}
\end{align}
is valid for the convex hull $\G$ defined by (\ref{equ:Gdefn}).
\end{theorem}

\noindent Note that both sides of (\ref{equ:general_inequality}) contain
the term $\gamma \nu \, q_x$, and so the presentation of both
(\ref{equ:general_inequality}) and (\ref{equ:main}) could be simplified.
However, we leave these slightly unsimplified so as to facilitate our
discussion in Section \ref{sec:chen_ineqs} below.

Let $\hat r$ be any scalar in $[0,r]$. Since $\hat r \le \|x\|$ is also
valid for $\F$, we can replace $r$ by $\hat r$ in (\ref{equ:main}) to
obtain an alternate inequality based on $\hat r$. In fact, considering
$\hat r$ to be variable in this inequality while all other quantities
are fixed, we see that the inequality is linear in $\hat r$, which implies
that all such valid inequalities over $\hat r \in [0,r]$ are actually
dominated (implied) by the two extremes $\hat r = 0$ and $\hat r = r$. We
summarize this observation in the following corollary.

\begin{corollary} \label{cor:main2}
Under the assumptions of Theorem \ref{the:main}, the infinite class
of inequalities gotten by replacing $r$ with $\hat r \in [0,r]$ is
dominated by the two inequalities (\ref{equ:main}) and
\begin{align}
&\nu^2 \left(H_q \bullet X + 2 \, g_q^T x + f_q \right)
+ \nu \left( 2 \, g_l b^T \bullet X + (f_l b - 2 \alpha g_l)^T x - \alpha f_l \right) \nonumber \\
&\quad\quad \ge \qllowbd \tr(X) - \left( 2 \, g_l c^T \bullet X + f_l c^T x \right) - \rho \nu (2 \, g_l^T x + f_l). \label{equ:main2}
\end{align}
corresponding to the extremes $\hat r = r$ and $\hat r = 0$, respectively.
\end{corollary}

\subsection{Example: Slab inequalities} \label{sec:slab_ineqs}

In this subsection, we introduce a specialization of our inequalities,
which we will return to in Section \ref{sec:nonnegF}.

Suppose that we have knowledge of $s \in \R^n$ and $\lambda,\mu \in \R$
such that
\begin{equation} \label{equ:slab}
    \F \subseteq {\cal S} := \{ x : \lambda \le s^T x \le \mu \},
\end{equation}
i.e., every $x \in \F$ satisfies $\lambda \le s^T x \le \mu$. We call
${\cal S}$ a valid {\em slab\/} and, abusing notation, we refer to
${\cal S}$ by its tuple \( (\lambda,s,\mu) \). For example, since \(
\F \) is bounded, for any vector \( s \) with $\|s\| = 1$, choosing \(
\lambda = -\nu \) and \( \mu = \nu \) yields a valid slab. Given any
slab $(\lambda, s, \mu)$, we discuss two choices of nonnegative $q_x$
and $l_x$.

First, define $q_x := \mu - s^Tx \ge 0$ and $l_x := s^T x - \lambda \ge
0$. Note that \( q_x \) is linear in this case, and \(\qllowbd =
q_x + l_x = \mu - \lambda\). Then
\eqref{equ:general_inequality} becomes
\begin{align} 
    (\gamma + \nu) \nu (\mu - s^T x) + &(\gamma + \nu) (s^T x - \lambda) (b^T x - \alpha) \nonumber \\
                                       &\ge (\mu - \lambda) (x^T x + \gamma \nu)  - (s^T x - \lambda) c^T x - \rho \nu (s^T x - \lambda). \label{equ:slabineq}
\end{align}
Alternatively, we could also take $q_x := s^T x - \lambda$ and $l_x :=
\mu - s^T x$ to obtain another, similar quadratic inequality.

Second, given the slab $(\lambda,s,\mu)$, we may assume without loss of
generality that $\lambda + \mu \ge 0$ and $\lambda^2 \le \mu^2$. To see
this, we consider three cases. First, if both $\lambda, \mu \ge 0$, then
the statement is clear. Second, if both $\lambda, \mu \le 0$, we can use
instead the equivalent representation of ${\cal S}$ by $-\mu \le -s^T x
\le -\lambda$. Finally, if $\lambda < 0$ and $\mu \ge 0$ with $\lambda
+ \mu < 0$, then we can likewise use $(-\mu,-s,-\lambda)$ instead. Now,
with $\lambda + \mu \ge 0$ and $\lambda^2 \le \mu^2$, we then define
$q_x := \mu^2 - (s^T x)^2 \ge 0$ and $l_x := (\lambda + \mu)(s^T x -
\lambda) \ge 0$ so that
\begin{align*}
    q_x + l_x
    &= \mu^2 - (s^T x)^2 + (\lambda + \mu) s^T x - \lambda\mu - \lambda^2 \\
    &= \mu^2 + (\mu - s^T x)(s^T x - \lambda) - \lambda^2 \\
    &\ge \mu^2 + 0 - \lambda^2 \ge 0.
\end{align*}
Hence, we obtain \eqref{equ:general_inequality} with $[q+l]_{\min} :=
\mu^2 - \lambda^2 \ge 0$.

\subsection{Example: Special case $c = 0$, $\alpha = 0$, and $\lambda \ge 0$} \label{sec:chen_ineqs}

In this subsection, we derive two cuts---see (\ref{equ:ourchen})
below---that are closely related to the cuts just discussed in Section
\ref{sec:slab_ineqs}, and these will play a special role in Section
\ref{sec:opf}. We assume $c = 0$ and $a = 0$, and we will use a slab
$(\lambda, s, \mu)$ with $\lambda \ge 0$. Note that $c = 0$ implies
$\rho = 0$.

For the first cut, consider the inequality
\eqref{equ:general_inequality} with \( c = 0 \) and $a = 0$, which is
further relaxed on the right-hand side:
\begin{align} 
    (\gamma + \nu) \nu q_x + (\gamma + \nu) l_x b^T x 
    &\ge \qllowbd x^T x + \gamma \nu \left( q_x + l_x\right) \nonumber \\
    &\ge \qllowbd (x^T x + \gamma \nu). \label{equ:chen_ineq1}
\end{align}
For the second cut, we consider a pair of functions \( l_x:=l(x) \)
and \( p_x:=p(x) \) that satisfy a different relationship than the
previously considered \( l_x \) and \( q_x \). Specifically, we assume
linear $l_x \ge 0$ and quadratic $p_x \ge 0$, and we require $l_x - p_x
\ge 0$ for all $x \in \F$ as well. We also define $[l - p]_{\min} \ge
0$ to be the minimum value of $l_x - p_x$ over $\F$. Then we have the
following result.

\begin{proposition}
Suppose $c = 0$, $a = 0$, and $l_x := l(x)$ and $p_x := p(x)$ are
nonnegative functions on $\F$ such that $l_x - p_x$ is also nonnegative
on $\F$. Then, for all $x\in\F$
\[
    {l_x b^T x - \gamma p_x \choose (l_x - p_x) x} \in \SOC.
\]
\end{proposition}

\begin{proof}
$(l_x - p_x) \|x\| = l_x \|x\| - p_x \|x\| \le l_x b^T x - \gamma p_x.$
\end{proof}

\noindent Using this proposition, the self-duality of the SOC, and
Proposition \ref{pro:scaledx}, we have
\[
    {\gamma + \nu \choose -(1 + \gamma \nu \|x\|^{-2})x}^T {l_x b^Tx 
    - \gamma p_x \choose (l_x - p_x) x} \ge 0,
\]
which rearranges and relaxes to
\begin{align}
    (\gamma + \nu) l_x b^T x - (\gamma + \nu) \gamma p_x
    &\ge \left( l_x - p_x \right) x^T x + \gamma \nu \left( l_x - p_x \right) \nonumber \\
    &\ge [l - p]_{\min} (x^T x + \gamma \nu ). \label{equ:chen_ineq2}
\end{align}
Note that (\ref{equ:chen_ineq2}) simplifies to $R \, l_x \, b^T x \ge
[l - p]_{\min} \, x^T x$ when $r = 0$, which is a consequence of the
simpler inequality $R \, b^T x \ge x^T x$; see (\ref{equ:basicineq0})
with $a = 0$. In other words, (\ref{equ:chen_ineq2}) appears to be
interesting only when $r > 0$.

We now consider a specific choice of \( q_x, l_x \), and \( p_x \) for
the inequalities \eqref{equ:chen_ineq1} and \eqref{equ:chen_ineq2} based
on the slab $0 \le \lambda \le s^T x \le \mu$. We choose $q_x := \mu^2
- (s^T x)^2$, $l_x := (\lambda + \mu) s^T x$, and $p_x := (s^Tx)^2 -
\lambda^2$ as the nonnegative functions, resulting in
\begin{align*}
    q_x + l_x = \mu^2 - (s^T x)^2 + (\lambda + \mu) s^T x  &\ge
    \mu^2 + \lambda \mu =: [q+l]_{\min} \\
    l_x - p_x = \lambda^2 - (s^Tx)^2 + (\lambda + \mu) s^Tx & \ge \lambda^2 +
    \lambda \mu =: [l-p]_{\min},
\end{align*}
where the inequalities follow from the RLT inequality $(\mu
- s^T x)(s^T x - \lambda) \ge 0$. Plugging these into
\eqref{equ:chen_ineq1}--\eqref{equ:chen_ineq2}, respectively, and
linearizing, we obtain
\begin{subequations} \label{equ:ourchen}
\begin{align} 
    (\gamma + \nu) \nu (\mu^2 - ss^T \bullet X) + (\gamma + \nu) (\lambda + \mu) sb^T \bullet X &\ge 
    (\mu^2 + \lambda\mu) (\tr(X) + \gamma \nu)
  \label{equ:ourchen1} \\
    (\gamma + \nu) (\lambda + \mu) sb^T \bullet X - (\gamma + \nu) \gamma ( ss^T \bullet X - \lambda^2) &\ge
    (\lambda^2 + \lambda\mu) (\tr(X) + \gamma \nu).
  \label{equ:ourchen2}
\end{align}
\end{subequations}

%% file: applications.tex
\section{Applications} \label{sec:applications}

In this section, we explore two applications of the inequalities
developed in Section \ref{sec:new_ineqs}. The first application shows
that the valid inequalities for the optimal power flow problem (OPF)
derived in \cite{Chen2017} are in fact just special cases of our
inequalities, whereas the derivation in \cite{Chen2017} was specifically
tailored to OPF. Our second application investigates the convex hull of
$\G$, where---departing from the form of (\ref{equ:qcqp_main})---$\F$
equals the intersection of the ball with the nonnegative orthant, i.e.,
$\F$ possesses polyhedral aspects as well. We study this form of $\F$
since it is relevant for any bounded feasible set with nonnegative
variables, where the bound is given by a Euclidean ball.

\subsection{Optimal power flow problem} \label{sec:opf}

In this subsection, we consider a result of Chen et al.~\cite{Chen2017},
which provides an exact formulation for the convex hull of a nonconvex,
quadratically constrained set appearing in the study of the optimal
power flow (OPF) problem. In particular, the authors added two new
linear inequalities to the Shor relaxation in order to capture the
convex hull. Whereas these two inequalities were specifically
derived for OPF, we will show that they are just special cases of
(\ref{equ:ourchen}) derived in Section \ref{sec:chen_ineqs}. For
additional background on convex relaxations of OPF, we refer the
reader to the two-part survey \cite{Low2014a,Low2014}. Here, we briefly mention
that the system \eqref{eq:chenetal} comes from considering a pair of adjacent
buses with their voltage magnitude constraints (\eqref{eq:chenetal:mag}) and
their voltage-angle difference constaints (\eqref{eq:chenetal:pad}); hence,
$L_{jj}$ and $U_{jj}$ are bounds on the voltage magnitude at bus
$j$ and $L_{ij}$ and $U_{ij}$ are bounds on the voltage-angle difference between
buses $i$ and $j$.

We restate the result of Chen et al.~using their notation. Let ${\cal
J}_C \subseteq \R^4$ be the convex hull of the following nonconvex
quadratic system:
\begin{subequations} \label{eq:chenetal}
\begin{align}
&L_{jj} \le W_{jj} \le U_{jj} \quad \forall \ j = 1,2 \label{eq:chenetal:mag} \\
&L_{12} W_{12} \le T_{12} \le U_{12} W_{12} \label{eq:chenetal:pad} \\
&W_{12} \ge 0 \label{eq:chenetal:nonneg} \\
&W_{11} W_{22} = W_{12}^2 + T_{12}^2 \label{eq:chenetal:rank1}
\end{align}
\end{subequations}
where the four variables are $(W_{11}, W_{22}, W_{12}, T_{12}) \in
\R^4$ and the data $L = (L_{11}, L_{22}, L_{12})$ and $U = (U_{11},
U_{22}, U_{12})$ satisfy $L \le U$ and $L_{jj} \ge 0$ for $j=1,2$.
Chen et al.'s interest in this particular convex hull arose from an
analysis of the OPF problem, where \eqref{eq:chenetal} appears as a
repeated substructure. As explained in \cite{Chen2017}, ${\cal
J}_C$ can alternatively be expressed as the following convex hull using
two complex variables $z_1, z_2 \in \C$:
\begin{equation} \label{equ:JC}
    {\cal J}_C = \clconv{
        \begin{pmatrix}
            z_1 z_1^* \\ z_2 z_2^* \\ \Re(z_1 z_2^*) \\ \Im(z_1 z_2^*)
        \end{pmatrix} \in \R^4
        :
        \begin{array}{c}
            L_{jj} \le z_j z_j^* \le U_{jj} \quad \forall \ j = 1,2 \\
            L_{12} \Re(z_1 z_2^*) \le \Im(z_1 z_2^*) \le U_{12} \Re(z_1 z_2^*) \\
            \Re(z_1z_2^*) \ge 0
        \end{array}
    }.
\end{equation}
In particular, equation (\ref{eq:chenetal:rank1}) is the usual
``rank-1'' condition, capturing the link between the linear variables
$(W_{11}, W_{22}, W_{12}, T_{12})$ and the quadratic expressions in
$z_1, z_2$. The authors proved that the pair of linear inequalities
\begin{subequations} \label{equ:chen}
\begin{align}
    \pi_0 + \pi_1 W_{11} + \pi_2 W_{22} + \pi_3 W_{12} + \pi_4 T_{12} \ge U_{22} W_{11} + U_{11} W_{22} - U_{11} U_{22} \label{equ:chen1} \\
    \pi_0 + \pi_1 W_{11} + \pi_2 W_{22} + \pi_3 W_{12} + \pi_4 T_{12} \ge L_{22} W_{11} + L_{11} W_{22} - L_{11} L_{22} \label{equ:chen2}
\end{align}
\end{subequations}
are valid for ${\cal J}_C$, where
\begin{align*}
    \pi_0 &:= -\sqrt{L_{11} L_{22} U_{11} U_{22}} \\
    \pi_1 &:= -\sqrt{L_{22} U_{22}} \\
    \pi_2 &:= -\sqrt{L_{11} U_{11}} \\
    \pi_3 &:= \left( \sqrt{L_{11}} + \sqrt{U_{11}} \right) \left( \sqrt{L_{22}} + \sqrt{U_{22}} \right)
                \frac{ 1 - f(L_{12}) f(U_{12}) }{ 1 + f(L_{12}) f(U_{12}) } \\
    \pi_4 &:= \left( \sqrt{L_{11}} + \sqrt{U_{11}} \right) \left( \sqrt{L_{22}} + \sqrt{U_{22}} \right)
                \frac{ f(L_{12}) + f(U_{12}) }{ 1 + f(L_{12}) f(U_{12}) }
\end{align*}
and where $f(x) := (\sqrt{1 + x^2} - 1)/x$ when $x > 0$ and $f(0) :=
0$. In fact, they proved that (\ref{equ:chen}), when added to the Shor
relaxation, is sufficient to capture ${\cal J}_C$:
\[
    {\cal J}_C
    =
    \left\{
    (W_{11}, W_{22}, W_{12}, T_{12}) 
        :
        \begin{array}{c}
        \text{ (\ref{eq:chenetal:mag})--(\ref{eq:chenetal:nonneg}) } \\
        W_{11} W_{22} \ge W_{12}^2 + T_{12}^2 \\
        \text{ (\ref{equ:chen}) }
        \end{array}
    \right\}.
\]
Here, the convex constraint $W_{11} W_{22} \ge W_{12}^2 + T_{12}^2$
is equivalent to the regular positive-semidefinite condition.

We now relate (\ref{equ:chen}) to our inequalities (\ref{equ:ourchen}).
Defining
\begin{equation} \label{eq:opf_feas_definition}
  \F := \left\{ x \in \R^3 :
      \begin{array}{c}
          L_{11} \le x_1^2 + x_2^2 \le U_{11} \\
          L_{22} \le x_3^2 \le U_{22} \\
          L_{12} x_1 x_3 \le x_2 x_3 \le U_{12} x_1 x_3 \\
          x_1 x_3 \ge 0, \quad x_3 \ge 0
      \end{array}
  \right\}.
\end{equation}
and $\G$ by (\ref{equ:Gdefn}), the following proposition establishes an
equivalence between ${\cal J}_C$ and $\G$.

\begin{proposition} \label{pro:JC=G}
${\cal J}_C = \{ (X_{11} + X_{22}, X_{33}, X_{13}, X_{23}) : (x,X) \in \G \}$.
\end{proposition}

\begin{proof}
Consider (\ref{equ:JC}). Because the quadratic terms $z_1 z_1^*$, $z_2
z_2^*$, and $z_1 z_2^*$ are unaffected by a rotation of $\C$ applied
simultaneously to both $z_1$ and $z_2$, we may enforce $\Re(z_2) \ge 0$
and $\Im(z_2) = 0$ without changing the definition of ${\cal J}_C$. Then
writing $z_1 = x_1 + i x_2$ and $z_2 = x_3$ for $x \in \R^3$, we thus
have ${\cal J}_C = \clconv{ (x_1^2 + x_2^2, x_3^2, x_1 x_3, x_2 x_3) : x
\in \F \subseteq \R^3 }$, which proves the proposition.
\end{proof}

\noindent Our next proposition establishes an alternative form for $\F$,
which matches the development in Section \ref{sec:new_ineqs} except that
the SOCs involve only two scalar variables, even though $\F$ is
3-dimensional. However, the results of Section \ref{sec:new_ineqs} can
easily be adapted to this case, the key point being that the Hessians of
the SOCs are equal. First we need a lemma.

\begin{lemma} \label{lem:2dimset}
For $n=2$, let ${\cal P} := \{ x \in \R^2 : Ax \le 0 \}$ be a polyhedral cone
with $A \in \R^{2 \times 2}$. Then ${\cal P} = \{ x : \| {x_1 \choose
x_2} \| \le b^T x \}$ for some $b \in \R^2$.
\end{lemma}

\begin{proof}
First assume that ${\cal P}$ is contained in the right side of the
plane, i.e., ${\cal P} \subseteq \{ x : x_1 \ge 0 \}$ and that ${\cal
P}$ is symmetric about the $x_1$ axis. Then, for some $\beta \ge 0$,
\begin{align*}
    {\cal P} &= \{ x : x_1 \ge 0, -\beta x_1 \le x_2 \le \beta x_1 \} \\
      &= \{ x : x_1 \ge 0, x_2^2 \le \beta^2 x_1^2 \} \\
      &= \{ x : x_1 \ge 0, x_1^2 + x_2^2 \le (1 + \beta^2) x_1^2 \} \\
      &= \{ x : \| \textstyle{{x_1 \choose x_2}} \| \le \sqrt{1 + \beta^2} \, x_1 \},
\end{align*}
which proves the result in this case. For general ${\cal P}$, we may apply
an orthogonal rotation to revert to the previous case, which does not
affect the norm $\|{x_1 \choose x_2} \|$ (but does change the exact form
of $b$).
\end{proof}

\noindent We next state and prove the proposition. Note that the
assumptions \( L_{22} > 0 \) and \( U_{12} > L_{12} \) in the
proposition are realistic for power networks: the first ensures the
voltage magnitude at a bus is positive, and the second allows for a
positive voltage-angle difference between the involved buses.

\begin{proposition} \label{pro:altF}
Suppose $L_{22} > 0$ and $U_{12} > L_{12}$. Then the feasible set, $\F$, defined by
\eqref{eq:opf_feas_definition} satisfies
\[
    \F = \left\{ x \in \R^3 :
        \begin{array}{c}
        \sqrt{L_{11}} \le \left\| {x_1 \choose x_2} \right\| \le \sqrt{U_{11}} \\
        \left\| {x_1 \choose x_2} \right\| \le b_1 x_1 + b_2 x_2 \\
        \sqrt{L_{22}} \le x_3 \le \sqrt{U_{22}}
        \end{array}
    \right\}
\]
where $b_1$ and $b_2$ uniquely solve the system
\[
    \begin{pmatrix} 1 & L_{12} \\ 1 & U_{12} \end{pmatrix}
    \begin{pmatrix} b_1 \\ b_2 \end{pmatrix} =
    \begin{pmatrix} \sqrt{1 + L_{12}^2} \\ \sqrt{1 + U_{12}^2} \end{pmatrix}.
\]
\end{proposition}

\begin{proof}
The assumption $L_{22} > 0$ implies $x_3 > 0$, which in turn implies
\[
    \F = \left\{ x \in \R^3 :
        \begin{array}{c}
            L_{11} \le x_1^2 + x_2^2 \le U_{11} \\
            \sqrt{L_{22}} \le x_3 \le \sqrt{U_{22}} \\
            L_{12} x_1 \le x_2 \le U_{12} x_1 \\
            x_1 \ge 0
        \end{array}
    \right\}.
\]
Next, the assumption $U_{12} > L_{12}$ makes $x_1 \ge 0$ redundant, and
clearly the first constraint in $\F$ is equivalent to $\sqrt{L_{11}} \le
\| {x_1 \choose x_2} \| \le \sqrt{U_{11}}$.

To complete the proof, we claim that $L_{12} x_1 \le x_2 \le U_{12} x_1$
is equivalent to the SOC constraint $\| {x_1 \choose x_2} \| \le b_1
x_1 + b_2 x_2$. Indeed, it is clear that the set defined by these two
linear inequalities is a polyhedral cone with the two extreme rays $r^1
= {1 \choose L_{12}}$ and $r^2 = {1 \choose U_{12}}$. So, by Lemma
\ref{lem:2dimset}, the set is
SOC-representable in the form $\| {x_1 \choose x_2} \| \le b_1 x_1 + b_2
x_2$ for some $b \in \R^2$. In particular, the extreme rays $r^j$ must
satisfy $\| r^j \| = b^T r^j$. By plugging in the values of $r^1$ and
$r^2$, we get the $2 \times 2$ linear system defining $b$, as desired.
Note that the $2 \times 2$ matrix is invertible because its determinant
$U_{12} - L_{12}$ is positive.
\end{proof}

Based on Propositions \ref{pro:JC=G} and \ref{pro:altF}, we now prove
that (\ref{equ:chen}) is simply (\ref{equ:ourchen}) tailored to the OPF
case.

\begin{theorem}
Inequalities (\ref{equ:chen}) are the inequalities (\ref{equ:ourchen})
tailored to system (\ref{eq:chenetal}).
\end{theorem}

\begin{proof}
By Proposition \ref{pro:JC=G}, we can translate (\ref{equ:chen1}) to the
variables $(x,X)$. After collecting terms, (\ref{equ:chen1}) becomes
\begin{equation} \label{eq:chenetal:new1}
    (\pi_0 + U_{11} U_{22}) + (\pi_1 - U_{22}) (X_{11} + X_{22}) +
    (\pi_2 - U_{11}) X_{33} + \pi_3 X_{13} + \pi_4 X_{23} \ge 0.
\end{equation}
Using Proposition \ref{pro:altF}, consider \eqref{equ:ourchen1} with
the following replacements:
\[
    \label{eq:data_in_chen_notation}
    x \leftarrow {x_1 \choose x_2}, \enskip
  \gamma \leftarrow \sqrt{L_{11}}, \enskip
  \nu \leftarrow \sqrt{U_{11}}, \enskip
  \lambda \leftarrow \sqrt{L_{22}}, \enskip
  s^T x \leftarrow x_3, \enskip
  \mu \leftarrow \sqrt{U_{22}}.
\]
This results in the following valid inequality:
\begin{align*}
    &\left(\sqrt{L_{22}U_{22}}+ U_{22}\right)\frac{X_{11} + X_{22} +
    \sqrt{L_{11}U_{11}}}{\sqrt{L_{11}} + \sqrt{U_{11}}} \le \\
    &\quad\quad \left(\sqrt{L_{22}} +
    \sqrt{U_{22}}\right)
    (b_1 X_{13} + b_2 X_{23}) + (U_{22} - X_{33})\sqrt{U_{11}}.
\end{align*}
Simple, although tedious, algebraic manipulations establish that
this inequality is precisely (\ref{eq:chenetal:new1}). A similar
argument establishes that (\ref{equ:chen2}) corresponds to
(\ref{equ:ourchen2}).\footnote{We provide Matlab code for these
manipulations in the file \texttt{chenetal/verify\_chenetal.m} at the website \url{https://github.com/A-Eltved/strengthened_sdr}.}
\end{proof}

\noindent We also verified numerically that
(\ref{equ:chen}) is not captured by $\Shor \cap \KSOC$ in this case.

\subsection{Intersection of the ball and nonnegative orthant} \label{sec:nonnegF}

As stated in the Introduction, the critical feature of $\F$ studied
in this paper is its intersection of the ball with a second
SOC-representable set, which shares the Hessian identity matrix.
However, there are of course many other forms of $\F$ that can be of
interest in practice. For example, when $\F$ is the nonnegative orthant,
then $\G$ is the completely positive cone, which can be used to model
many NP-hard problems as
linear conic programs \cite{Burer2009}. Another common case is when $\F$
is a box, e.g., the set $[0,1]^n$ \cite{Burer2009a}.

Let us examine the case in which $\F$ is the intersection of the
nonnegative orthant and the unit ball. For general $n$, define $\F :=
\{ x \ge 0 : \| x \| \le 1 \} \subseteq \R^n$. Since $$x \in \F \quad
\Rightarrow \quad \|x\| \le \|x\|_1 = e^T x,$$ we have
\begin{equation} \label{equ:nonnegF}
    \F \subseteq \{ x : \|x\| \le 1, \|x\| \le e^T x \},
\end{equation}
and for $n = 2$, one can actually show that (\ref{equ:nonnegF}) is
an equation. Since $\F$ is a subset of the nonnegative orthant, any
inequality, which is valid for the completely positive cone, is
also valid for $\F$, but here we focus on the implied structure in
(\ref{equ:nonnegF}). Section \ref{sec:new_ineqs} applies with $r = 0,
R = 1, c = 0, b = e$, and $\alpha = 0$. In particular, the constraints
$\tr(X) \le 1$ and $\tr(X) \le e^T x$ are valid for $\G$; see the
Introduction and inequality (\ref{equ:basicineq0}).

We can strengthen $\tr(X) \le 1$ and $\tr(X) \le e^T x$ using the slab
inequalities of Section \ref{sec:slab_ineqs}. Geometrically, given
any $s \in \R^n$ with $s \ge 0$ and $\|s\| = 1$, we have the slab
$\lambda := 0 \le s^T x \le 1 =: \mu$, which is valid for $\F$:
\[
    0 \le s^T x \le \|s\| \|x\| = \|x\| \le 1.
\]
After
linearization, inequality (\ref{equ:slabineq}) in this case reads $1 -
s^T x + s^T Xe \ge \tr(X)$. Moreover, if we switch the role of $q_x$ and
$l_x$ in (\ref{equ:slabineq})---recall that $q_x$ is linear for slabs---then
we have $s^T x + e^T x - s^T Xe \ge \tr(X)$. Rearranging, we write
these two inequalities as
\begin{subequations} \label{equ:nonneg}
\begin{align}
    &\tr(X) \le 1 + s^T (Xe - x) \label{equ:nonneg:1} \\
    &\tr(X) \le e^T x - s^T (Xe - x). \label{equ:nonneg:2}
\end{align}
\end{subequations}
Letting $s$ vary over its constraints $\|s\| = 1$ and $s \ge 0$, we
derive a compact SOC-representation of this class of inequalities over
various domains of $\G$.

\begin{theorem} \label{the:nonnegF}
Let $(I,J)$ be a partition of the index set $\{1,\ldots,n\}$, and define the
domain
\[
    {\cal D}_{IJ} := \left\{ (x,X) :
        \begin{array}{l}
        [Xe - x]_I \ge 0 \\ \relax
        [Xe - x]_J \le 0 
        \end{array}
    \right\}.
\]
Then the following SOC constraints are locally valid for $\G$ on
${\cal D}_{IJ}$:
\begin{subequations} \label{equ:nonnegsoc}
    \begin{align}
    &\tr(X) \le 1 - \|[Xe - x]_J\| \label{equ:nonnegsoc:1} \\
    &\tr(X) \le e^T x - \|[Xe - x]_I\|. \label{equ:nonnegsoc:2}
    \end{align}
\end{subequations}
Moreover, (\ref{equ:nonnegsoc}) imply all valid inequalities
(\ref{equ:slabineq}) derived from slabs of the form $0 \le s^T x \le 1$,
where $s$ is any vector satisfying $\|s\| = 1$ and $s \ge 0$.
\end{theorem}

\begin{proof}
Consider the constraints (\ref{equ:nonneg}), and for notational
convenience, define $y := Xe - x$. Because $s \ge 0$, the quantity $s^T
y$ on the right-hand side of (\ref{equ:nonneg:1}) breaks into $s_I^T
y_I \ge 0$ and $s_J^T y_J \le 0$ on ${\cal D}_{IJ}$. By minimizing
the right-hand side of (\ref{equ:nonneg:1}) with respect to $s$, we
achieve the tightest cut corresponding to $s = (s_I, s_J) = (0, -y_J /
\|y_J\|)$, which yields $\tr(X) \le 1 - \|y_J\|$, as desired. A similar
argument for (\ref{equ:nonneg:2}) yields $\tr(X) \le e^T x - \|y_I\|$.
\end{proof}

\noindent We remark that, when $I$ is empty, inequality
(\ref{equ:nonnegsoc:2}) reduces to the inequality $\tr(X) \le
e^T x$ over $D_{IJ}$. Similarly, when $J$ is empty,
(\ref{equ:nonnegsoc:1}) is $\tr(X) \le 1$.

In practice, one idea for using Theorem \ref{the:nonnegF} is as follows.
For a given relaxation in $(x,X)$, solve the relaxation to obtain
an optimal solution $(\bar x, \bar X)$. Then define the partition
$(I,J)$ and corresponding domain $D_{IJ}$ according to $\bar X e -
\bar x$. Then, if either of the inequalities in (\ref{equ:nonnegsoc})
is violated, we can derive a violated supporting hyperplane of the SOC
constraint. After adding the violated linear inequality to the current
relaxation, which is globally valid because it is linear, we can resolve
and repeat the process.

We close this section with an example showing that the cuts derived
above are not implied by $\Shor \cap \KSOC$.

\begin{example}
Let $n=2$, and consider $I = \{1,2\}$ and $J = \emptyset$. Then $\tr(X)
\le e^T x - \| Xe - x\|$ is valid on the domain ${\cal D}_{IJ} = \{
(x,X) : Xe - x \ge 0 \}$. In particular, $\tr(X) \le e^T x - u^T (Xe -
x)$ for all vectors $u$ satisfying $\|u\| = 1$, and taking $u = e_1$, we
have $\tr(X) \le e^T x - [Xe - x]_1$, which is globally valid since
it is linear. Minimizing $e^T x - [Xe - x]_1 - \tr(X)$ over $\Shor \cap
\KSOC$ yields the optimal value $-0.088562$, indicating that $\Shor \cap
\KSOC$ does not capture this valid constraint.
\end{example}

%% file: separation.tex
\section{Separation} \label{sec:separation}

In this section, we argue that the inequalities
(\ref{equ:main})--(\ref{equ:main2}) given by Theorem \ref{the:main}
and Corollary \ref{cor:main2} are separable in polynomial time. To
state this result precisely, we assume that $\rho$ has already been
pre-computed and that a fixed convex relaxation of the convex hull $\G$
defined by (\ref{equ:Gdefn}) is available. For convenience, we write this
fixed convex relaxation
\[
    \Rc := \left\{ (x,X) : Y(x,X) \in \widehat\Rc \right\} \supseteq \G,
\]
where $Y(x,X)$ is given by (\ref{equ:Ydefn}) and $\widehat\Rc$ is a
closed, convex cone in the space of $(n+1) \times (n+1)$ symmetric
matrices. In particular, $\Rc$ is just the slice of $\widehat\Rc$
with the top-left corner of $Y$ set to 1. Then the relaxation of
(\ref{equ:qcqp_main}) over $\Rc$ can be stated as $\min\{ H \bullet X +
2 \, g^T x : (x,X) \in \Rc \}$ with dual
\[
    \max \left\{ y : \begin{pmatrix} -y & g^T \\ g & H \end{pmatrix}  \in \widehat\Rc^*\right\}
\]
where $\widehat\Rc^*$ is the dual cone of $\widehat\Rc$. We state this
general form for ease of notation and to make evident that one can choose
different $\Rc$ in computation. For example, one could take $\Rc
= \Shor$ at one extreme or $\Rc = \Shor \cap \KSOC$ at the other.

In fact, to separate (\ref{equ:main})--(\ref{equ:main2}) we will
use the following observation concerning $\Rc$, $\widehat\Rc$, and
$\widehat\Rc^*$:

\begin{observation}
Given a quadratic function $q(x) := x^T H_q x + 2 \, g_q^T x + f_q$, if
there exists $y \in \R$ such that
\[
    \begin{pmatrix}
        -y + f_q & g_q^T \\ g_q & H_q
    \end{pmatrix}
    \in \widehat\Rc^*,
\]
then $q(x) \ge y$ for all $x \in \F$.
\end{observation}

\noindent This observation follows by weak duality because $y$ is
a lower bound on the optimal relaxation value of $H_q \bullet X +
2 \, g_q^T x + f_q$ over $(x,X) \in \Rc$, which is itself a lower
bound on the minimum value of $q(x)$ over $x \in \F$. As a result,
the following system guarantees that the conditions of
Theorem \ref{the:main} on $q(x)$ and $l(x)$ hold, where $(H_q,g_q,f_q)$,
$(g_l,f_l)$, and $[q+l]_{\min}$ are the variables:
\begin{subequations} \label{equ:dualsys}
\begin{align}
    &\begin{pmatrix}
        f_q & g_q^T \\ g_q & H_q
    \end{pmatrix}
    \in \widehat\Rc^*, \quad
    \begin{pmatrix}
        f_l & g_l^T \\ g_l & 0
    \end{pmatrix}
    \in \widehat\Rc^*, \label{equ:dualsys:linear} \\
    &[q+l]_{\min} \ge 0, \quad
    \begin{pmatrix}
        -[q+l]_{\min} + f_q + f_l & (g_q + g_l)^T \\ g_q + g_l & H_q
    \end{pmatrix}
    \in \widehat\Rc^*. \label{equ:dualsys:min}
\end{align}
\end{subequations}

Then, separation amounts to optimizing the linear function in
(\ref{equ:main})---or (\ref{equ:main2}) as the case may be---over
(\ref{equ:dualsys}) for fixed values of $(x, X)$. However, before we
state the exact separation problem for (\ref{equ:main}), we require one additional assumption,
namely that $\F$ is full-dimensional, i.e., there exists $\hat x
\in \F$ such that $\|\hat x\| < R$ and $\|\hat x - c\| < b^T \hat x -
\alpha$. In this case, it is well known that $\G$ and hence $\Rc$
are also full-dimensional in $(x,X)$-space. In
particular, $(\hat x, \hat x \hat x^T) \in \text{int}(\G) \subseteq
\text{int}(\Rc)$, and hence
\[
    \hat Y := {1 \choose \hat x}{1 \choose \hat x}^T \in \text{int}(\widehat\Rc).
\]
It thus follows that $\widehat\Rc^* \cap \{ J : \hat Y \bullet J \le 1
\}$ is a bounded truncation of $\widehat\Rc^*$.\footnote{Indeed,
for any closed, convex cone ${\cal K}$ and dual cone ${\cal K}^*$, given
$\hat x \in \text{int}({\cal K})$, we claim the truncation ${\cal
K}^* \cap \{ s : \hat x^T s \le 1 \}$ is bounded. Specifically, its
recession cone ${\cal K}^* \cap \{ s : \hat x^T s \le 0 \} = \{0\}$.
If not, then some nonzero $\tilde s \in {\cal K}^*$ satisfies $\hat x^T
\tilde s \le 0$. Because $\hat x$ is interior, for sufficiently small
$\epsilon > 0$, the point $\tilde x := \hat x - \epsilon s$ satisfies
$\tilde x \in {\cal K}$ and $\tilde x^T \tilde s < 0$. However, this
contradicts the fact that $\tilde s \in {\cal K}^*$.}
This truncation is important so that the separation problem presented
below has a bounded feasible set and thus has a well-defined optimal
value.

We are now ready to state the separation subproblem for (\ref{equ:main})
given fixed values $(\bar x, \bar X)$ of the variables $(x,X)$:
\begin{subequations} \label{equ:separation_problem}
\begin{align}
    \min \quad
    &(\gamma + \nu) \nu \left(H_q \bullet \bar X + 2 \, g_q^T \bar x + f_q \right)
    + (\gamma + \nu) \left( 2 \, g_l b^T \bullet \bar X + (f_l b - 2 \alpha g_l)^T \bar x - \alpha f_l \right) \\
    & \quad\quad - \qllowbd \tr(\bar X) - \gamma\nu \left( H_q \bullet \bar X + 2 (g_q + g_l)^T \bar x + (f_q + f_l) \right) \nonumber \\
    & \quad\quad + \left( 2 \, g_l c^T \bullet \bar X + f_l c^T \bar x \right) + \rho \nu (2 \, g_l^T \bar x + f_l) \\
    \st \quad &(\ref{equ:dualsys}) \\
              &\hat Y \bullet \begin{pmatrix} f_q & g_q^T \\ g_q & H_q \end{pmatrix} \le 1, \quad
               \hat Y \bullet \begin{pmatrix} f_l & g_l^T \\ g_l & 0 \end{pmatrix} \le 1.
\end{align}
\end{subequations}
The subproblem for (\ref{equ:main2}) is similar---just replace $r$ with
0.

We remark that system (\ref{equ:dualsys}) could be simplified in
certain cases. For example, if $r = 0$ and hence $\F$ is convex,
then it is not difficult to see that the second condition of
(\ref{equ:dualsys:linear}), which ensures that $l(x)$ is nonnegative
over $\F$, could be replaced by a dual system based on $\F$ alone, not
on $\Rc$. One could also simplify by forcing additional structure on
$q(x)$ and $l(x)$. For example, one could separate against the slabs
$\lambda \le s^T x \le \mu$ introduced in Section \ref{sec:slab_ineqs}
by forcing $(H_q, g_q, f_q) = (0, - \frac12 s, \mu)$, $(g_l,f_l) =
(\frac12 s, -\lambda)$, and $[q+l]_{\min} = \mu - \lambda$, in which case
(\ref{equ:dualsys:min}) is automatically satisfied.

The following example demonstrates the separation procedure, whose
implementation will be discussed in the next section:

\begin{example} \label{exa:closed_gap}
  Consider the $2$-dimensional problem
  \begin{align*}
    \min \ \ & -x_1^2 - x_2^2 - 1.1 x_1 - x_2 \\
    \text{s.t.} \ \ & \|x\| \le 1 \\
                  & \|x\| \le 1 - x_1 - x_2
  \end{align*}
  with \( H = -I \), \( g = (-0.55, -0.5) \), \( r = 0 \), \(
  R = 1 \), \( a = -1 \), \( b = (-1,-1) \), and \( c = (0,0) \) in
  \eqref{equ:qcqp_main}. All values reported here are truncated from the
  computations and therefore approximate. The optimal value of \( \min\{H \bullet X + 2g^T x : (x,X)
  \in \Shor\cap\KSOC \} \) is \( -1.1431 \) with optimal solution
  \[
  \bar x = \begin{pmatrix*}[r] 0.2922 \\  -0.1783\end{pmatrix*}, \
  \bar X = \begin{pmatrix*}[r] 0.4963 &  -0.3210 \\ -0.3210 & 0.5037\end{pmatrix*}.
  \]
  Solving the separation subproblem at $(\bar x, \bar X)$, we obtain the cut
  corresponding to
  \begin{align*}
      q_1(x) &= x^T \begin{pmatrix} -0.3812 & 0 \\ 0 & -0.3812 \end{pmatrix} x +
        2 \begin{pmatrix} -0.5578 \\ -0.5531 \end{pmatrix}^T x + 0.8563,
        \\ l_1(x) &= 2 \begin{pmatrix} 0.3462 \\ 0.3608 \end{pmatrix}^T x + 1, \\
        [q_1+l_1]_{\min} &= 1.42.
  \end{align*}
  We add the corresponding cut, resolve to obtain a new $(\bar x, \bar X)$, and
  repeat this loop two more times, resulting in the cuts
  \begin{align*}
      q_2(x) &= x^T \begin{pmatrix*}[r] -0.7065 & 0.1719 \\ 0.1719 & -0.4368 \end{pmatrix*} x +
        2 \begin{pmatrix} -0.7808 \\ -0.7278 \end{pmatrix}^T x + 1,
        \\ l_2(x) &= 2 \begin{pmatrix} 0.3442 \\ 0.3626 \end{pmatrix}^T x + 1, \\
        [q_2+l_2]_{\min} &= 1.155, \\
        q_3(x) &= x^T \begin{pmatrix*}[r] -0.6296 & 0.2398 \\ 0.2398 & -0.4512 \end{pmatrix*} x +
        2 \begin{pmatrix} -0.7868 \\ -0.7580 \end{pmatrix}^T x + 1,
        \\ l_3(x) &= 2 \begin{pmatrix} 0.3479 \\ 0.3591 \end{pmatrix}^T x + 1, \\
        [q_3+l_3]_{\min} &= 1.149.
    \end{align*}
    We finally obtain the rank-1, and hence optimal, solution
  \[ Y(x^\star,X^\star) =
      \begin{pmatrix*}[r] 1 & 0.7071 & -0.7071 \\ 0.7071 & 0.5 & -0.5 \\
    -0.7071 & -0.5 & 0.5 \end{pmatrix*}
  \]
  with objective value \( -1.0707 \). We note that, even though the
  procedure generates three cuts, the last cut is actually enough to
  recover the rank-1 solution. Moreover, running this procedure starting
  from $\Shor$ instead of $\Shor \cap \KSOC$, we also get the same
  optimal $(x^\star, X^\star)$ after adding 16 cuts.



\end{example}

%% file: numerical_examples.tex
\section{Computational Results} \label{sec:computation}

To quantify the practical effect of the cuts proposed in Theorem
\ref{the:main} and Corollary \ref{cor:main2}, we embed the
separation subproblem described in Section \ref{sec:separation} in a
straightforward implementation to solve random instances of the form
\eqref{equ:qcqp_main}. We consider two relaxations to ``bootstrap'' the
separation procedure: \( \Shor \) and \( \Shor\cap\KSOC \). We will
denote by \( \Cuts \) the points \( (x,X) \) satisfying the added cuts,
so that our improved relaxations will be expressed as \( \Shor\cap\Cuts
\) and \( \Shor\cap\KSOC\cap\Cuts \).

We implement our experiments in Matlab 9.6 (R2019a) using CVX \cite{cvx}
to model the relaxations and MOSEK 9.1 \cite{mosek} to solve them. We run the problem
instances on a single core of an Intel Xeon E5-2650v4 processor using a
maximum of 2GB memory. We do not report
complete run times because we are most interested in the strength of the
added cuts, but we do report the number of cuts added to measure the
overall effort. Recall that calculating a single cut requires solving
the separation problem (\ref{equ:separation_problem}) described in
Section \ref{sec:separation}, which in essence involves three copies of
the current bootstrap relaxation---$\Shor \cap \Cuts$ or $\Shor \cap
\KSOC \cap \Cuts$---since the separation problem includes three sets of variables
constrained to be in the dual cone of the bootstrap relaxation as described by
\eqref{equ:dualsys}. However, to give the reader a sense of the run times,
consider the following: for an instance of our largest dimension, \(
n=10 \), solving \( \Shor \) took approximately 0.6 seconds, solving
\( \Shor\cap\KSOC \) required about 50 seconds, and solving a single
separation problem for \( \Shor\cap\KSOC \) took approximately 64 seconds. We note that our
implementation is rudimentary and makes no effort to take 
advantage of, for example, any particular problem structure or sparsity, so these times
can probably be improved significantly.

We generate a single random instance by fixing the dimension \( n
\) and generating random data \( a,b,c,r,R,H,g \) in such a way that
(\ref{equ:qcqp_main}) is feasible with a known interior point
$\hat x$, which is also randomly generated. In short, we first set
$R = 1$ without loss of generality, generate $r$ uniformly in $[0,R]$,
generate $\hat x$ uniformly in $\{ x : r \le \hat x \le R \}$,
generate $b,c,H,g$ with entries i.i.d.~standard normal, and finally
set $a := b^T \hat x - \|\hat x - c\| - \theta$, where $\theta$
is uniform in $[0,1]$ so that $\F$ has a nonempty interior.\footnote{We refer the reader to our
GitHub site (\url{https://github.com/A-Eltved/strengthened_sdr}) for
the full random-generation procedure.} Recall that $\hat x$ is
required for the separation procedure as discussed in Section
\ref{sec:separation}. Before running the separation procedure for an instance, we compute $\rho
$ by a binary search on $\rho$ over the interval $[0, \|c\|]$
as discussed in Section \ref{sec:new_ineqs}.
Then, when running the overall algorithm, we consider
the current relaxation's optimal solution \( (\bar x, \bar X) \) to be
{\em separated} if: the objective value of the separation subproblem
\eqref{equ:separation_problem} is less than \( \tau_{\text{sep}} =
-10^{-5} \); or the optimal value of the separation subproblem for the
inequalities \eqref{equ:main2} in Corollary \ref{cor:main2}, i.e., 
\eqref{equ:separation_problem} with $r = 0$, 
is less than  $ \tau_{\text{sep}} $.
If $(\bar x, \bar X)$ is indeed separated, we add the resulting
cut represented by the data \( (H_q, g_q, f_q, g_l, f_l, [q+l]_{\min})
\) to the current bootstrap relaxation, optimize for a new point to be
separated, and repeat. The overall loop stops when the current $(\bar x,
\bar X)$ is not separated with tolerance $\tau_{\text{sep}}$.

Regarding a given relaxation and its optimal solution $(\bar x, \bar X)$, we
say the relaxation is {\em exact\/} if \( Y(\bar x, \bar X) \) satisfies
\begin{equation} \label{eq:numerical_rank-1_criterium}
  \frac{\lambda_1(Y(\bar x, \bar X))}{\lambda_2(Y(\bar x, \bar X))}
> \tau_{\text{rank}},
\end{equation}
where \( \lambda_1(M) \) denotes the largest
eigenvalue of \( M \), \( \lambda_2(M) \) denotes the second largest
eigenvalue of \( M \), and \( \tau_{\text{rank}} > 0 \) is a tolerance,
which we choose to be \( 10^4 \) in our implementation, ensuring that
$Y(\bar x, \bar X)$ is numerically rank-1. We define the {\em gap\/} as
the difference between the optimal value of \eqref{equ:qcqp_main} and
the relaxation optimal value. Note that an exact relaxation implies 
a gap of 0.

After running the algorithm on a particular instance, we classify the
instance into one of two categories: {\em exact initial\/} or {\em
inexact initial\/}, when the initial bootstrap relaxation is exact
or inexact, respectively. Furthermore, we break all inexact-initial
instances into one of three subcategories: {\em improved\/}, when the
initial relaxation gap is improved but not completely closed to 0; {\em
closed\/}, when the relaxation becomes exact after adding one or more
cuts (i.e, the resulting $(\bar x, \bar X)$ satisfies
\eqref{eq:numerical_rank-1_criterium}); and {\em no improvement\/}, when no cuts are successfully added to
improve the gap, i.e., the separation routine does not help. (Actually,
in the tables below, we will not directly report information about the
exact-initial and no-improvement instances, as these details will be
implicitly available from the other categories.)

We conduct these experiments for several values of $n$ and many randomly
generated instances. In addition, we also consider special cases
where some of the data \( a,b,c,r,R \) is fixed to zero in order to
assess whether the cuts are more effective in these special cases. In
particular, we consider the following three cases: the general case,
where no data is fixed {\em a priori\/} to zero; the special case
with \( r=a=0 \) and \(c=0 \); and the case of the TTRS (two
trust region subproblem) with \( r=0 \) and \( b=0 \). For each of these cases, we
generate 15,000 instances for each dimension \( 2 \le n \le 10 \), and
we solve each instance twice, once bootstrapping from \( \Shor \) and
once from \( \Shor\cap\KSOC \).

For the improved and closed instances, we report the average number
of cuts added. Also for the improved instances, we report the average
gap closure in percentage terms, i.e., we report the average relative
gap closure. Since we do not actually know the optimal value of
(\ref{equ:qcqp_main}) for the improved instances, to approximate the
relative gap closure from above, we calculate a local minimum value,
\( v_{\text{local}} \), by taking the lowest value of the quadratic
objective function gotten by running Matlab's \texttt{fmincon} with
100 random initial points. The relative gap for the instance is then
calculated as \[ \text{relative gap closure} = \frac{v_{\text{relax
final}} - v_{\text{relax initial}}}{v_{\text{local}} - v_{\text{relax
initial}}} \times 100\%, \] where \( v_{\text{relax initial}} \) is the
optimal value of the initial relaxation and \( v_{\text{relax final}} \)
is the optimal value of the final relaxation.
Note that a larger gap closure corresponds to a stronger relaxation,
i.e., a larger gap closure is better.



\subsection{The general case} \label{sec:numexp_shorksoc_general}

We consider 15,000 random instances for each dimension \( 2\le n\le
10 \) and report the results separately for the \( \Shor \) and \(
\Shor\cap\KSOC \) bootstrap relaxations in Tables \ref{tab:gen_shor}
and \ref{tab:gen_ksoc}, respectively.

In Table \ref{tab:gen_shor}, we see that our cuts improve the \( \Shor
\) relaxation in many instances. For \( n = 2 \), it improves more than
a third of the inexact instances, and it closes the gap for about 9\%.
As the dimension goes up, these proportions go down, suggesting that
our cuts are more effective in lower dimensions.

\begin{table}[ht]
  \centering
  \begin{tabular}{r|r|rrr|rr}
      \( n \) & Inexact initial & Improved & Avg cuts & Avg gap closure & Closed & Avg cuts \\ \hline
      2 &  2923 &  1188 &  15 &  51\% &   264 &   4 \\
      3 &  2582 &   761 &  17 &  46\% &   175 &   7 \\
      4 &  2161 &   422 &  10 &  40\% &    53 &   7 \\
      5 &  1801 &   416 &  10 &  36\% &    46 &   9 \\
      6 &  1583 &   265 &  12 &  36\% &    29 &   8 \\
      7 &  1360 &   186 &  11 &  36\% &    10 &  11 \\
      8 &  1091 &   140 &  14 &  39\% &    15 &   7 \\
      9 &  1029 &   107 &  12 &  34\% &     4 &  15 \\
     10 &   896 &    86 &  13 &  30\% &     4 &  11
  \end{tabular}
  \caption{Results for the \( \Shor \) bootstrap relaxation on 15,000
  random general instances for each dimension $n$. The columns {\em
  Inexact initial\/}, {\em Improved\/}, and {\em Closed\/} report the
  number of instances out of 15,000 in each category.}
  \label{tab:gen_shor}
\end{table}

Table \ref{tab:gen_ksoc} shows that $ \Shor\cap\KSOC $ is generally
quite strong for instances of the form \eqref{equ:qcqp_main}. Especially
for larger \( n \), the number of inexact instances is small, and the
ability of our cuts to improve or close the gaps is limited. 
In particular, for \( n \ge 4 \) our cuts do not improve any of the inexact
instances, which again suggests that the cuts are most
helpful in lower dimensions.

\begin{table}[ht]
  \centering
  \begin{tabular}{r|r|rrr|rr}
      \( n \) & Inexact initial & Improved & Avg cuts & Avg gap closure & Closed & Avg cuts \\ \hline
      2 &   251 &    40 &  13 &  45\% &     3 &   3 \\
      3 &    84 &     5 &  36 &  48\% &     0 & --- \\
      4 &    44 &     0 & --- &   --- &     0 & --- \\
      5 &    16 &     0 & --- &   --- &     0 & --- \\
      6 &     6 &     0 & --- &   --- &     0 & --- \\
      7 &     7 &     0 & --- &   --- &     0 & --- \\
      8 &     2 &     0 & --- &   --- &     0 & --- \\
      9 &     3 &     0 & --- &   --- &     0 & --- \\
     10 &     3 &     0 & --- &   --- &     0 & --- 
  \end{tabular}
  \caption{Results for the \( \Shor\cap\KSOC \) bootstrap relaxation
  on the same 15,000 random general instances as depicted in Table
  \ref{tab:gen_shor} for each dimension $n$. The columns {\em
  Inexact initial\/}, {\em Improved\/}, and {\em Closed\/} report the
  number of instances out of 15,000 in each category.}
  \label{tab:gen_ksoc}
\end{table}

\subsection{Special case: $r = a = 0$ and $c = 0$}

We next consider the special case when $\F$ equals \( \{ x \in \R^n :
\|x\| \le 1, \|x\| \le b^T x\}\) with \( b \in \R^n \). Note that, by
rotating the feasible space, we may assume without loss of generality
that $b$ lies in the direction of $e$, the all ones vector. In
particular, we generate instances with \( b = \beta e \), where
\( \beta \in [1/\sqrt{n}, 1/\sqrt{n} + 2n] \). The choice of this interval for $\beta$ is based
on the following observation: for $\beta < 1/\sqrt{n}$ the feasible space $\F$
is empty; for $\beta = 1/\sqrt{n}$ the feasible space $\F$
has no interior; for $\beta \to \infty$, the constraint
$\|x\| \le b^T x$ resembles the half space $0 \le e^T x$.

Similar to Tables \ref{tab:gen_shor}--\ref{tab:gen_ksoc} of the
previous subsection, Tables \ref{tab:wedge_shor}--\ref{tab:wedge_ksoc}
contain the results of our separation algorithm on 15,000 randomly
generated instances for each dimension, where Table \ref{tab:wedge_shor}
corresponds to $\Shor$ and Table \ref{tab:wedge_ksoc} to
$\Shor\cap\KSOC$. Contrary to what we saw in the general case in Tables
\ref{tab:gen_shor}--\ref{tab:gen_ksoc}, there does {\em not\/} seem to
be a drop in the proportion of instances where the cuts help as \( n \)
increases. Overall, our cuts seem to be quite effective in this special
case.

\begin{table}[ht]
  \centering
  \begin{tabular}{r|r|rrr|rr}
      \( n \) & Inexact initial & Improved & Avg cuts & Avg gap closure & Closed & Avg cuts \\ \hline
    2 &  7744 &  2755 & 22 &  82\% &  4988 & 2 \\
    3 &  7635 &   914 & 23 &  86\% &  6495 & 3 \\
    4 &  7736 &   395 & 13 &  83\% &  6966 & 3 \\
    5 &  7709 &   401 &  4 &  81\% &  6596 & 3 \\
    6 &  7584 &   402 &  5 &  67\% &  7182 & 3 \\
    7 &  7648 &   185 &  5 &  87\% &  7463 & 3 \\
    8 &  7614 &   131 &  8 &  89\% &  7483 & 3 \\
    9 &  7566 &    77 &  7 &  93\% &  7489 & 2 \\
   10 &  7552 &    44 &  7 &  89\% &  7508 & 2
  \end{tabular}
  \caption{Results for the \( \Shor \) bootstrap relaxation on 15,000
  random instances with \( r=a=0 \) and \( c=0 \) for each dimension
  $n$. The columns {\em
  Inexact initial\/}, {\em Improved\/}, and {\em Closed\/} report the
  number of instances out of 15,000 in each category.}
  \label{tab:wedge_shor}
\end{table}

\begin{table}[ht]
  \centering
  \begin{tabular}{r|r|rrr|rr}
      \( n \) & Inexact initial & Improved & Avg cuts & Avg gap closure & Closed & Avg cuts \\ \hline
    2 &    15 &     0 & --- &   --- &    15 & 2 \\
    3 &    50 &     7 &  43 &  37\% &    30 & 2 \\
    4 &    36 &     4 &  78 &  75\% &    28 & 2 \\
    5 &    29 &     0 & --- &   --- &    27 & 3 \\
    6 &    15 &     3 &   8 &  88\% &    12 & 3 \\
    7 &    13 &     2 &   4 &  57\% &    11 & 2 \\
    8 &    12 &     0 & --- &   --- &    12 & 2 \\
    9 &     6 &     0 & --- &   --- &     5 & 1 \\
   10 &     6 &     0 & --- &   --- &     5 & 3
  \end{tabular}
  \caption{Results for the \( \Shor\cap\KSOC \) bootstrap relaxation
  on the same 15,000 random instances as depicted in Table
  \ref{tab:wedge_shor} with \( r=a=0 \) and \( c=0 \) for each dimension
  $n$. The columns {\em
  Inexact initial\/}, {\em Improved\/}, and {\em Closed\/} report the
  number of instances out of 15,000 in each category.}
  \label{tab:wedge_ksoc}
\end{table}

Specifically for $n = 2$, the results in Table \ref{tab:wedge_ksoc}
suggest that \( \Shor\cap\KSOC\cap\Cuts \) is tight, i.e., it captures
the convex hull $\G$. To test this further, we generated an additional
110,000 instances with \( n=2 \). The \( \Shor\cap\KSOC \) relaxation
was exact for 109,938 of these, and our cuts closed the gap for
the remaining 62 instances with an average of 3 cuts added. Our
computational experience thus motivates a conjecture:

\begin{conjecture}
For the 2-dimensional feasible space \( \F := \{ x \in \R^2 : \|x\|
\le 1, \|x\| \le b^T x \} \) with arbitrary \( b \in \R^2 \), \(
\Shor\cap\KSOC\cap\Cuts \) equals the convex hull \( {\cal G} \) defined
in \eqref{equ:Gdefn}.
\end{conjecture}

\noindent In addition, in Section \ref{sec:nonnegF}, for $n = 2$ and
$b = e$, we proposed the locally valid cuts (\ref{equ:nonnegsoc}),
which were derived from slabs of a particular form. (Note that these
cuts would not necessarily be valid for a different scaling $b =
\beta e$.) By generating many random objectives, we were able to
find 100 additional instances, which were {\em not\/} solved exactly
by \( \Shor\cap\KSOC \), and then separated just these locally valid
cuts---instead of the more general cuts represented by $\Cuts$. 
All 100 instances were solved exactly, i.e., achieved the tolerance $\tau_{\text{rank}}$. 
We believe this is strong evidence to support the following conjecture as
well:

\begin{conjecture}
For the 2-dimensional feasible space \( \F := \{ x \in \R^2 : \|x\| \le
1, \|x\| \le e^T x \} = \{ x \ge 0 : \|x\| \le 1 \} \), the constraints
defined by \( \Shor\cap\KSOC\) intersected with the locally valid cuts
(\ref{equ:nonnegsoc}) capture the convex hull \( {\cal G} \) defined in
\eqref{equ:Gdefn}.
\end{conjecture}

\subsection{Special case: TTRS (\( b=0 \) and \( r=0 \))}

Setting \( b = 0 \) and \( r = 0 \) in \eqref{equ:qcqp_main} with
\( a < 0 \) to ensure feasibility, we explore the two-trust-region
subproblem (TTRS). We generate 15,000 random instances of this type
for each dimension $ 2 \le n \le 10 $ and bootstrap from the $
\Shor $ and \( \Shor\cap\KSOC \) relaxations. The results are shown
in Tables \ref{tab:ttrs_shor} and \ref{tab:ttrs_ksoc}. The trends
in these tables are similar to what we saw in the general case in
Section~\ref{sec:numexp_shorksoc_general}. In particular, our cuts are
less effective in higher dimensions.

\begin{table}[ht]
  \centering
  \begin{tabular}{r|r|rrr|rr}
      \( n \) & Inexact initial & Improved & Avg cuts & Avg gap closure & Closed & Avg cuts \\ \hline
    2 &  1404 &   364 &  16 &  33\% &    86 &   4 \\
    3 &  1287 &   172 &  15 &  27\% &    34 &   4 \\
    4 &   985 &    79 &  12 &  27\% &    20 &   5 \\
    5 &   745 &    34 &   9 &  22\% &     7 &   3 \\
    6 &   508 &    14 &   7 &  22\% &     3 &   2 \\
    7 &   454 &     4 &   5 &  25\% &     2 &   3 \\
    8 &   347 &     5 &   8 &  58\% &     0 & --- \\
    9 &   293 &     0 & --- &   --- &     1 &   2 \\
   10 &   251 &     1 &   4 &   2\% &     0 & ---
  \end{tabular}
  \caption{Results for the \( \Shor \) bootstrap relaxation on 15,000
  random TTRS instances for each dimension $n$. The columns {\em
  Inexact initial\/}, {\em Improved\/}, and {\em Closed\/} report the
  number of instances out of 15,000 in each category.}
  \label{tab:ttrs_shor}
\end{table}
\begin{table}[ht]
  \centering
  \begin{tabular}{r|r|rrr|rr}
      \( n \) & Inexact initial & Improved & Avg cuts & Avg gap closure & Closed & Avg cuts \\ \hline
    2 &    31 &     4 &  20 &  24\% &     0 & --- \\
    3 &    78 &     7 &  43 &  29\% &     1 &   7 \\
    4 &    63 &     3 &  55 &  19\% &     0 & --- \\
    5 &    34 &     1 &  59 &   6\% &     0 & --- \\
    6 &    22 &     0 & --- &   --- &     0 & --- \\
    7 &    16 &     0 & --- &   --- &     0 & --- \\
    8 &    14 &     0 & --- &   --- &     0 & --- \\
    9 &     6 &     0 & --- &   --- &     0 & --- \\
   10 &     4 &     0 & --- &   --- &     0 & ---
  \end{tabular}
  \caption{Results for the \( \Shor\cap\KSOC \) bootstrap relaxation
  on the same 15,000 random TTRS instances as depicted in Table
  \ref{tab:ttrs_shor} for each dimension $n$. The columns {\em
  Inexact initial\/}, {\em Improved\/}, and {\em Closed\/} report the
  number of instances out of 15,000 in each category.}
  \label{tab:ttrs_ksoc}
\end{table}

We catalog the following example showing an explicit case for $n=2$ in
which our cuts close the gap for TTRS compared to just applying $\Shor \cap
\KSOC$.

\begin{example}
  Consider the instance with $n=2$, $r = 0$, $R = 1$, $a = -0.77$, and
  \[ 
  b = \begin{pmatrix} 0 \\ 0 \end{pmatrix}, \quad
  c = \begin{pmatrix*}[r] -0.38 \\ 0.18 \end{pmatrix*}, \quad
  H = \begin{pmatrix*}[r] -1.32 & 0.21 \\  0.21 & -0.81 \end{pmatrix*}, \quad
  g = \begin{pmatrix*}[r] -0.25 \\ 0.05 \end{pmatrix*}.
  \]
  The (approximate) optimal value of \( \min\{ H \bullet X + 2 \, g^T x : (x,X) \in
  \Shor\cap\KSOC \} \) is $-0.9087$ and the solution is not rank-1.
  Solving the separation problem starting from this relaxation, we
  obtain the (approximate) cut corresponding to
  \[ 
      g_l = \begin{pmatrix*}[r] 1.8633 \\ -0.8826 \end{pmatrix*}, \quad
    f_l = 4.1236, \quad
    [q+l]_{\min} = 1.2604,
\]
\[
    H_q = \begin{pmatrix*}[r] -4.9035 & 0.0000 \\ 0.0000 & -4.9035 \end{pmatrix*}, \quad
    g_q = \begin{pmatrix*}[r] -1.8633 \\ 0.8826 \end{pmatrix*}, \quad
    f_q = 2.0403.
  \]
  Solving the relaxation with this cut, results in the (numerically) rank-1 solution
  \[ 
  Y(x^\star, X^\star) =
  \begin{pmatrix*}[r]
     1.0000 &  -0.9065 &   0.4223 \\
    -0.9065 &   0.8217 &  -0.3828 \\
     0.4223 &  -0.3828 &   0.1783 
  \end{pmatrix*}
  \]
  with (approximate) optimal value $-0.8943$.
\end{example}

%% file: conclusions.tex
\section{Conclusions} \label{sec:conclusions}

In this paper, we have derived a new class of valid linear inequalities
for SDP relaxations of problem \eqref{equ:qcqp_main}. These cuts are
separable in polynomial time, which, by the equivalence of separation
and optimization (up to an $\epsilon > 0$ optimality tolerance), ensures that the SDP relaxation enforcing all of these
inequalities is polynomial-time solvable. We have also shown that a
special case of our cuts has been applied by Chen et al.~\cite{Chen2017}
to obtain the convex hull of an important substructure arising in the
OPF problem. In addition, we have extended our methodology to derive
new, locally valid, second-order-cone cuts for nonconvex quadratric
programs over the mixed polyhedral-conic set $\{ x \ge 0 : \|x\| \le 1
\}$. Using specific examples as well as computational experiments, we
have demonstrated that the new class of valid inequalities strengthens
the strongest known SDP relaxation, \( \Shor\cap\KSOC \), especially in
low dimensions.

For the specific 2-dimensional feasible set \( \F = \{ x \in \R^2 :
\|x\| \le 1, x \le b^T x \} \), our computational experiments indicate
that our cuts intersected with $\Shor\cap\KSOC$ capture the relevant
convex hull $\G$. We leave this as a conjecture requiring further
research. Furthermore, when $b = e$, we also conjecture that the locally
valid cuts (\ref{equ:nonnegsoc}), which are derived from slabs, are by
themself enough to capture $\G$. For general $\F$, however, our cuts do
not close the gap fully, and so there remains room for improvement.

One limitation of our approach is the assumption that the SOC constraint
\eqref{sub:qcqp_cons_soc} shares the identity Hessian with the hollow
ball \eqref{sub:qcqp_cons_quad}. If instead we are presented with a
general SOC constraint $\| J x - c \| \le b^T x - a$, where $J \in \R^{n
\times n}$ is arbitrary, one idea would be to bound
\begin{align*}
b^T x - a
&\ge \| J x - c \| \\
&\ge \| x - c \| - \| x - Jx \| \\
&= \| x - c \| - \| (I - J) x \| \\
&\ge \| x - c \| - \sqrt{ \lambda_{\max}[(I - J)^T (I - J)]} \, R,
\end{align*}
which yields the valid constraint $\|x - c\| \le b^T x - \left(a -
\sqrt{\lambda_{\max} [(I - J)^T (I - J)]} R\right)$, to which our
methodology can be applied. Additional options for handling arbitrary
Hessians can be considered by refining the deriviations of Section
\ref{sec:new_ineqs}.

Further opportunities for future research include streamlining the
separation subroutine, investigating the effectiveness of our cuts
in higher dimensions, and examining other applications where the
structure of \eqref{equ:qcqp_main} appears. Also, the idea of using the
self-duality of a cone to derive valid linear cuts could be applied to
other self-dual cones or possibly even non-self-dual cones.

\section*{Acknowledgments}

The authors acknowledge the support of their respective universities,
which allowed the first author to visit the second author in 2019--20,
when this research was initiated. In addition, the authors wish
to thank Dan Bienstock for constructive comments as well as two
anonymous referees, whose suggestions improved the paper significantly.